\documentclass[11pt]{article}
\usepackage{amsmath, amssymb, amscd, amsthm, amsfonts}
\usepackage{graphicx}

\usepackage[utf8]{inputenc} 
\usepackage[T1]{fontenc} 
\usepackage[colorlinks=true, pdfstartview=FitV, linkcolor=blue,
citecolor=blue, urlcolor=blue]{hyperref}

 \let\mathscr\relax
\usepackage[scr]{rsfso}
\usepackage{mathrsfs}
\usepackage{multicol}
\usepackage{stmaryrd}
\usepackage{verbatim}

\renewcommand{\d}{\text{d}} 
 
\newcommand{\Rd}{\mathbb{R}^d}

\newcommand{\E}{\mathbb{E}} 
\renewcommand{\P}{\mathbb{P}} 
\newcommand{\F}{\mathscr{F}} 

\newcommand{\q}[1]{\langle #1 \rangle_t} 

\newtheorem{theorem}{Theorem}[section]

\newtheorem{lemma}[theorem]{Lemma}
\newtheorem{proposition}[theorem]{Proposition}
\newtheorem{conjecture}[theorem]{Conjecture}
\newtheorem{corollary}[theorem]{Corollary}

\theoremstyle{definition}
\newtheorem{definition}[theorem]{Definition}

\theoremstyle{remark}
\newtheorem{remark}[theorem]{Remark}

\title{Tail Asymptotics of the Signature of various stochastic processes and its connection to the Quadratic Variation}
\author{Martin Albert Gbúr\footnote{University of Warwick, UK. This research has been funded by the Undergraduate Research Support Scheme (URSS). The author is very grateful to Horatio Boedihardjo for an interesting project, frequent meetings and enlightening discussions. The author is also grateful to Tessy Papavasiliou for introducing him to RPT and useful research ideas.}}
\date{Project supervisors: Horatio Boedihardjo and Anastasia Papavasiliou}

\begin{document}

\maketitle

\begin{abstract}
The signature of a path is a sequence, whose $n$-th term contains $n$-th order iterated integrals of the path. These iterated integrals of sample paths of stochastic processes arise naturally when studying solutions of differential equation driven by those processes. This paper extends the work of Boedihardjo and Geng \cite{Boedihardjo:2019td} who found a connection between the signature of $d$-dimensional Brownian motion and the time elapsed, which equals to the quadratic variation of Brownian motion. We prove this connection for a more general class of processes, namely for  Itô signature of semimartingales. We also establish a connection between the fWIS Signature of fractional Brownian motion and its Hurst parameter. We propose a conjecture for extension to multidimensional space.
\end{abstract}

\section{Introduction}\label{section-introduction}

\subsection{Theoretical background}\label{subsection-background}

This is a very brief introduction to the Rough path theory introduced in \cite{Lyons:1998td}, roughly following Chapter 3 of \cite{Lyons:2007td}. Let $V$ be a real Banach space. Define the space of formal sequences of tensors by $T((V)) = \prod_{k=0}^{\infty} V^{\otimes k}$ and the truncated tensor algebra of order $N$ by $T^N(V) = \bigoplus_{k=0}^N V^{\otimes k}$. 

For a path $X:[0,T] \rightarrow V$ of finite $p$-variation for some $p<2$ we define the signature of $X|_{[s,t]}$ as 
\begin{align*}
  \mathbb{X}_{s,t} &= (1, \mathbb{X}_{s,t}^1, \mathbb{X}_{s,t}^2, \dots) \in T((V)) \hspace{1cm} \text{where} \\
  \mathbb{X}_{s,t}^n &= \int_s^t \int_s^{u_n} \dots \int_s^{u_2} \d X_{u_1} \otimes \dots \otimes \d X_{u_{n-1}} \otimes \d X_{u_n}  
\end{align*}
 and the integrals are defined in Young sense. In the case of a $d$-dimensional Brownian motion ($V = \Rd$ with basis $e_1, \dots, e_d$), we can define the signature similarly, by letting 
\begin{align*}
    \mathbb{B}_{s,t}^n :&= \sum_{1 \leq i_1, \dots, i_n \leq d} \mathbb{B}_{s,t}^{n, i_1, \dots, i_n}  e_{i_1} \otimes \dots \otimes e_{i_n} \in (\Rd)^{\otimes n} \hspace{1cm} \text{where} \\
    \mathbb{B}_{s,t}^{n, i_1, \dots, i_n} :&= \int_s^t \int_s^{u_n} \dots \int_s^{u_2} \d X_{u_1}^{i_1} \dots \d X_{u_{n-1}}^{i_{n-1}} \d X_{u_n}^{i_n} 
\end{align*}
However, as we're no longer in the finite $p$-variation setting, there are multiple ways of defining the integral (e.g. Itô, Stratonovich) resulting in different objects. 

Let $\Delta_T = \{(s,t) \in [0,T]^2: 0\leq s \leq t \leq T\}$. If we truncate the signature after its $n$-th term, we can view $\mathbb{X}$ as a map from $\Delta_T$ to $T^N(V)$, where (1) $\mathbb{X}^0 = 1$ and Chen's theorem asserts that (2) $\mathbb{X}_{s,u} \otimes \mathbb{X}_{u,t} = \mathbb{X}_{s,t}$ for any $s\leq u \leq t \in [0,T]$. The Itô and Stratonovich signatures of Brownian motion also share these two properties, and we call all such mappings \textit{Multiplicative functionals}. We can also extend the notion of $p$-variation to these functionals by defining the metric  
\begin{align*}
    d_p(\mathbb{X}, \mathbb{Y}) = \max_{1\leq i \leq n} \sup_{0=t_0 < t_1 < \dots < t_k = T} \left( \sum_{0\leq j \leq k-1} \| \mathbb{X}_{t_j,t_{j+1}}^i - \mathbb{X}_{t_j,t_{j+1}}^i \|_{V^{\otimes i}}^{p/i} \right)^{i/p}
\end{align*}
where $n$ is the integer part of $p$ and the supremum is taken over all partitions of $[0,T]$. We say that $\mathbb{X}$ has finite $p$-variation if $d_p(\mathbb{X},1) < \infty$. This leads us to the definition of a $p$-\textit{rough path} as a degree $\lfloor p \rfloor$ multiplicative functional with finite $p$-variation. Both Itô and Stratonovich signatures are $p$-rough paths for any $p\in(2,3)$.

There is a famous theorem by T. Lyons \cite{Lyons:1998td}, which states that any $p$-rough path can be uniquely extended into a multiplicative functional on $T((V))$, i.e. the signature of a finite $p$-variation path for $p\in(1,2)$ is the only multiplicative functional with the desirable property that $\mathbb{X}^1_{s,t} = X_t - X_s$. For Brownian motion, the choice is then determined by the second term. 

In statistical applications, we might want to infer properties of a process driven by a rough input, from its observations in time (see e.g. \cite{Fermanian:2023td} for a concise overview). Since the signature arises naturally in describing solutions to path-driven differential equations as Taylor expansions, it should encode a lot of the properties we're interested in. We would therefore like to lift the driving process into a suitable rough path - such that if we embed our observations into a path (continuous, of bounded variation) in some appropriate way and compute its signature, this will provide us with a good approximation of the rough path. 

This motivates a class of rough paths called $p$-\textit{geometric rough paths} defined as the closure of the space of signatures of bounded variation paths under the $p$-variation metric. In the case of Brownian motion, the Itô signature isn't a $p$-geometric rough path, but the Stratonovich signature is. Moreover, it equals the limit of signature of the dyadic piecewise linear interpolations to the Brownian sample path, and therefore we call it the cannonical lift. 

\subsection{Previous results}\label{subsection-previous}

Since the signature encodes properties of a path, the natural question is, how much information can we decode from the signature and if any properties get lost in the transformation. For bounded variation paths, it was shown that the signature determines a path up to tree-like equivalences \cite{Hambly:2010td}, and a similar result was proved for Brownian sample paths as well \cite{Qian:2011td}. It is therefore natural to ask, whether we can recover the quadratic variation of a sample path of a stochastic process from its signature. The following estimate on the $p$-variation of a $p$-rough path $\mathbb{X}$ is known:
\begin{align}\label{eq:1}
    \|\mathbb{X}^n\|_{\text{proj}} \leq \frac{C(\mathbb{X})^{n/p}}{(n/p)!},
\end{align}
where $C(\mathbb{X})$ is a constant depending on the $p$-variation of $\mathbb{X}$.

Let $V = \Rd$ and assume that $(B_t)_{t\geq 0} = (B_t^1, \dots, B_t^d)_{t\geq 0}$ is a $d$-dimensional Brownian motion with its cannonical lift into a $p$-geometric rough path $\mathbb{B}$ for $2<p<3$. In view of (\ref{eq:1}) it is natural to study the quantity: 
\begin{align}\label{eq:2}
    L^{\mathbb{B}}_{s,t} = \limsup_n \left\| \left(\frac{n}{2}\right)! \mathbb{B}_{s,t}^n \right\|_n^{\frac{2}{n}}
\end{align}
where $(\| \cdot \|_n)_n$ is an \textit{admissible} sequence of norms: 
\begin{itemize}
    \item $\|\cdot \|_n$ is a norm on $(\Rd)^{\otimes n}$
    \item for any $a\in (\Rd)^{\otimes n}$ and $b \in (\Rd)^{\otimes m}$ we have $\| a\otimes b \|_{m+n} \leq \| a \|_{n} \|  b \|_{m}$
    \item for any permutation $\sigma$ on $\{1, \dots n\}$, let $P_{\sigma}$ be a linear operator which sends $v_1 \otimes \dots \otimes v_n$ into $v_{\sigma (1)} \otimes \dots \otimes v_{\sigma (n)}$ and then we have for any $a\in (\Rd)^{\otimes n}$ that $\| P_{\sigma} a \|_{n} = \| a \|_{n}$
\end{itemize}
One example of such sequence of norms is the \textit{projective norm}:
\begin{align*}
    \| a \|_{n, \text{proj}} := \inf \left\{ \sum_l |v_1^l|_{\Rd} \dots |v_n^l|_{\Rd} \text{ when } a = \sum_l v_1^l \otimes \dots \otimes v_n^l \right\}
\end{align*}
which depends on the choice of norm on $\Rd$. In the case of $l^1$ we would get 
\begin{align*}
    \|\mathbb{B}_{s,t}^n\|_{n,l^1} = \sum_{i_1, \dots, i_n = 1}^d |\mathbb{B}_{s,t}^{n, i_1, \dots, i_n}|,
\end{align*}
which coincides with the $l^1$ norm on $(\Rd)^{\otimes n}$. Another admissible norm is $l^2$, under which 
\begin{align*}
    \|\mathbb{B}_{s,t}^n\|_{n,l^2} = \left( \sum_{i_1, \dots, i_n = 1}^d |\mathbb{B}_{s,t}^{n, i_1, \dots, i_n}|^2 \right)^{1/2}
\end{align*}
and then $(\Rd)^{\otimes n}$ is a Hilbert space.

In \cite{Boedihardjo:2019td} it was proved that with the set up as above, we have: 

\begin{theorem}
    Let $(B_t)_{t\geq 0} = (B_t^1, \dots, B_t^d)_{t\geq 0}$ be a $d$-dimensional Brownian motion, $(\| \cdot \|_n)_n$ an \textit{admissible} sequence of norms and $L^{\mathbb{B}}_{s,t}$ defined as in (\ref{eq:2}). If each element of the canonical basis $\{ e_1, \dots, e_d \}$ of $\mathbb{R}^d$ has norm one with respect to $\| \cdot \|$, then there exists a deterministic constant $\kappa_d$ depending on dimension $d$ and the choice of norms, such that almost surely: 
    \begin{align*}
        L^{\mathbb{B}}_{s,t} = \kappa_d (t-s) \hspace{1cm} \forall s<t
    \end{align*}
    Moreover, we have the upper estimate $\kappa_d \leq d^2$, and under the $l^p$ norm on $\Rd$ ($p\in [1,2]$) and the associated \textit{projective} tensor norms $(\| \cdot \|_n)_n$ we have the lower estimate $\kappa_d \geq (d-1)/2$
\end{theorem}

The aim of this paper is to extend this result by considering a wider class of integrator, different admissible norms, and prove similar results for the Itô signature.

\subsection{The plan of the paper}\label{subsection-previous}

In section \ref{section-main} we first focus on one-dimensional iterated Itô integrals and we use asymptotic behaviour of Hermite polynomials to get an exact formula for this limit. 

This approach is then further explored in section \ref{section-fBM}, where thanks to the connections between Hermite polynomials and Wiener integrals, we can extend this to the fractional Brownian Motion case and formulate an explicit relation between the Hurst parameter $H$ and its iterated fWIS integrals. 

Finally, we describe our endeavour to generalise our results into multiple dimensions in section \ref{section-multidim}, which we spent the most time working on, but due to significant difficulties we were able to show very little. We comment on our approaches and describe where this work could be taken further.

\section{Main result}\label{section-main}

First recall some basic facts from stochastic analysis, following the reference \cite{Klenke:2020td}: 

Fix a filtered probability space $(\Omega, \F, (\F_t)_{t\geq 0}, \P)$. A real valued process $(X_t)_{t\geq 0}$ adapted to $(\F_t)_{t\geq 0}$ is called a continuous semimartingale if it can be decomposed into $X_t = M_t + A_t$, where $(M_t)_{t\geq 0}$ is a continuous local martingale and $(A_t)_{t\geq 0}$ is a continuous bounded variation process started at zero. Local martingales are processes for which there exists a sequence of stopping times $\tau_n \rightarrow \infty$ a.s. such that the stopped processes $(M_{t\wedge \tau_n})_{t\geq 0}$ are all martingales. Let $\langle X \rangle_t$ be the quadratic variation of $X$. Then the following version of time-dependent Itô formula holds true:

\begin{lemma}
    Let $(X_t)_{t\geq 0}$ be a continuous semimartingale and $F\in C^{1,2}(\mathbb{R}^+ \times \mathbb{R})$. Then:
    \begin{align*}
        F(\q{X}, X_t) - F(0, X_0) = \int_0^t \partial_x F( \q{X}, X_t) \d X_t + \int_0^t \left( \partial_t + \frac{1}{2} \partial_x^2 \right) F( \q{X}, X_t) \d \q{X}, 
    \end{align*}
    or formally: 
    \begin{align*}
        \d F(\q{X}, X_t) =  \partial_x F( \q{X}, X_t) \d X_t + \left( \partial_t + \frac{1}{2} \partial_x^2 \right) F( \q{X}, X_t) \d \q{X}, 
    \end{align*}
    where by $\partial_t, \partial_x$ we mean differentiation with respect to the first and second coordinate respectively.
\end{lemma}

We would like to identify a sequence of polynomials $h_n$ such that 
\begin{align}
(\partial_t + \frac{1}{2} \partial_x^2) h_n( \q{X}, X_t) &= 0 \label{eq:3}\\
\partial_x h_n( \q{X}, X_t) &= nh_{n-1} \label{eq:4} \\ 
h_n( 0, 0) &= 0 \label{eq:5}
\end{align}
as then $H_t^n := h_n(\q{X}, X_t)$ satisfies the SDE: $\d H_t^n = nH_t^{n-1} \d X_t$. From here we can work inductively to express the $n$-th order iterated integral of $X$ in terms of $H_t^n$. If we further assume that $h_0(t,x) = 1$, then we need $h_1(t,x) = x$, so that $H_t^1 = X_t = \int_0^t \d X_t$. For the induction step, suppose that 
\begin{align*}
    H_s^n = n! \int_0^s \int_0^{s_n} \dots \int_0^{s_2} \d X_{s_1} \dots \d X_{s_{n-1}} \d X_{s_{n}}
\end{align*}
Then for any process started at zero we have: 
\begin{align*}
H_t^{n+1} = (n+1)\int_0^t H_s^n = (n+1)! \int_0^t \int_0^s \int_0^{s_n} \dots \int_0^{s_2} \d X_{s_1} \dots \d X_{s_{n-1}} \d X_{s_{n}} \d X_{s} 
\end{align*}
As it turns out, there exists a sequence of polynomials satisfying the equations (\ref{eq:3}), (\ref{eq:4}), (\ref{eq:5}):
\begin{definition} \label{hermite}
    Scaled Hermite polynomials are defined by 
    \begin{align*}
        \exp \left(\alpha x - \alpha^2 t/2 \right) = \sum_{n=0}^{\infty} \frac{\alpha^n}{n!}h_n(t,x),
    \end{align*}
    or more explicitly by $h_n(t,x) = (t/2)^{n/2}\Tilde{h}_n(x/\sqrt{t})$ where $\Tilde{h}_n(x) = (-1)^n e^{x^2/2} \frac{\d^n}{\d x^n}(e^{-x^2/2})$ are the Hermite polynomials.
\end{definition}

Thus we proved the following: 

\begin{proposition} \label{iterated}
    Let $(X_t)_{t\geq 0}$ be a continuous $\mathbb{R}$-valued semimartingale started at zero and $h_n$ the $n$-th scaled Hermite polynomial as in Lemma \ref{hermite}. Then for any $t\geq 0$ we have:
    \begin{align*}
        \mathbb{X}^{n,\text{Itô}}_{0,t} := \int_0^t \int_0^{s_n} \dots \int_0^{s_2} \d X_{s_1} \dots \d X_{s_{n-1}} \d X_{s_{n}} = \frac{1}{n!} h_n(\q{X}, X_t)
    \end{align*}
\end{proposition}

To analyse the asymptotic behaviour of $|\mathbb{X}^{n,\text{Itô}}_{0,t}|$ we'll exploit the asymptotic properties of Hermite polynomials as found in \cite{Abram:1964td}, p.508-510: 
\begin{align*}
    \Tilde{h}_n(x) \sim e^{x^2/2} \frac{2^n}{\sqrt{\pi}} \Gamma\left(\frac{n+1}{2}\right) \cos \left( x\sqrt{2n} - \frac{n\pi}{2} \right)
\end{align*}
Now we are in a good shape to prove the theorem (note that in $d=1$, the admissible tensor norms are just the absolute value on $\mathbb{R}$):

\begin{theorem}\label{main_thrm}
    Let $(X_t)_{t\geq 0}$ be a continuous $\mathbb{R}$-valued semimartingale started at zero. Then almost surely:
    \begin{align*}
        L^{\mathbb{X}, \text{Itô}}_{0,t} := \limsup_n \left| \left(\frac{n}{2}\right)! \mathbb{X}_{0,t}^{n,\text{Itô}} \right|^{2/n} = \frac{\q{X}}{2}
    \end{align*}
\end{theorem}

\begin{proof}
    For simplicity, denote $x = X_t$ and $t=\q{X}$, which are fixed throughout the proof. By Proposition \ref{iterated}, asymptotics of Hermite polynomials and Stirling's formula we have:
    \begin{align*}
        L^{\mathbb{X}, \text{Itô}}_{0,t} &= \limsup_n \left| \left(\frac{n}{2}\right)! \frac{1}{n!} \left(\frac{t}{2}\right)^{n/2} \Tilde{h}_n\left(\frac{x}{\sqrt{t}}\right)\right|^{2/n} \\
        &= \frac{t}{2} \limsup_n \left| \frac{\sqrt{2\pi n/2}(n/(2e))^{n/2}}{\sqrt{2\pi n} (n/e)^n}  e^{x^2/(4t)} \frac{2^n}{\sqrt{\pi}} \Gamma\left(\frac{n+1}{2}\right) \cos \left( x\sqrt{\frac{n}{t}} - \frac{n\pi}{2} \right) \right|^{2/n}
    \end{align*}
    So it remains to show that this limit is equal to 1. The first fraction can be rewritten as $(\frac{e}{2n})^{n/2}2^{-1/2}$ and asymptotically we also have:
    \begin{align*}
        \Gamma\left(\frac{n+1}{2}\right) = \Gamma\left(\frac{n}{2} + 1 - \frac{1}{2}\right) \sim \Gamma\left(\frac{n}{2} + 1\right) \left(\frac{n}{2} + 1\right) ^{-1/2} \sim \sqrt{\pi n} \left(\frac{n}{2e}\right)^{n/2} \left(\frac{n}{2} + 1\right) ^{-1/2} 
    \end{align*}
    Consequently:
    \begin{align*}
        L^{\mathbb{X}, \text{Itô}}_{0,t} &= \frac{t}{2} \limsup_n \left| 2^{-n-1/2} \sqrt{\pi n} \left(\frac{n}{2} + 1\right) ^{-1/2} \cdot  e^{x^2/(4t)} \frac{2^n}{\sqrt{\pi}} \cos \left( x\sqrt{\frac{n}{t}} - \frac{n\pi}{2} \right) \right|^{2/n} \\
         &= \frac{t}{2} \limsup_n \left( \left(\frac{n}{2}\right)^{1/n} \left(\frac{n}{2} + 1\right) ^{-1/n}   e^{x^2/(2tn)}  \right) \left| \cos \left( x\sqrt{\frac{n}{t}} - \frac{n\pi}{2} \right) \right|^{2/n}
    \end{align*}
    Now the RHS converges to 1, except the case when $\cos \left( x\sqrt{\frac{n}{t}} - \frac{n\pi}{2} \right) \rightarrow 0$ fast enough. Let us suppose for contradiction that it does. By continuity of cosine, the argument must approach the set $\cos^{-1}(\{ 0 \}) = \{ k\pi + \frac{\pi}{2}: k\in \mathbb{Z}\} =: \frac{\pi}{2} + \pi\mathbb{Z}$. We claim that if $x\sqrt{\frac{n}{t}} \rightarrow \frac{\pi}{2} + \pi\mathbb{Z}$, then there exists $k\in \mathbb{Z}$ such that $x\sqrt{\frac{n}{t}} \rightarrow \frac{\pi}{2} + k$. To see this, consider the difference 
    \begin{align*}
        \left|x\sqrt{\frac{n+1}{t}} - x\sqrt{\frac{n}{t}}\right| = \left|\frac{x}{\sqrt{t}}\right| (\sqrt{n+1}-\sqrt{n}) = \left|\frac{x}{\sqrt{t}}\right|     \frac{1}{\sqrt{n+1}+\sqrt{n}} \rightarrow 0,
    \end{align*}
    and so this is eventually smaller than $\frac{\pi}{2}$. But this yields a contradiction, since $\frac{x}{\sqrt{t}}\sqrt{n} \rightarrow \infty$ as $n\rightarrow \infty$, which concludes the proof.
\end{proof}

\begin{remark}\label{rem}
    There is a much nicer proof of this result, which however doesn't give us the precise factor, by which $L^{\mathbb{X}}_{0,t}$ depends on $\q{X}$. It uses the fact that each local martingale can be written as a Brownian motion running at a different speed of time, and then since the signature is invariant under time-reparametrisation, $L^{\mathbb{X}}_{0,t} = L^{\mathbb{B}}_{0,\q{X}} = \kappa_1 \q{X}$. This can be extended to diffusions with drift using Girsanov's theorem. 
\end{remark}

\section{Extension into the fractional Brownian motion case}\label{section-fBM}

In real-world applications of stochastic analysis, we're often interested in modelling phenomena which depend on a random driving signal. The simplest model for this signal is Brownian motion. However, most real inputs exhibit some form of temporal dependence, and so it's sometimes convenient to replace the driving signal by a generalisation of Brownian motion that depends on a parameter $H\in (0,1)$, which describes the time-persistence of the process. For example, to model the water level in a river we would expect $H > 1/2$, but to model prices of electricity in a liberated market we would expect $H < 1/2$. In statistical applications, it could be of interest to estimate the parameter $H$ and to derive an expression of $H$ in terms of the signature of $B^H$ will be the aim of this section. 

\begin{definition}
Fractional Brownian Motion (fBM) of Hurst parameter $H\in (0,1)$ is a centered Gaussian process $B^H: \Omega \times [0,\infty ) \rightarrow \mathbb{R}$ with a.s. continuous paths and covariance 
\begin{align*}
    R_H(s,t) := E[B^H_sB^H_t] =  \frac{1}{2}(t^{2H} + s^{2H} - |t-s|^{2H})
\end{align*}
\end{definition} 

Except when $H = 1/2$ and the process is just the standard BM, this isn't a semimartingale, hence the result of the previous section doesn't apply, and it's not even clear how to set up a consistent integration theory. Not only there are multiple ways of doing this, but these also depend on the value of $H$. Although for any $H$ we can construct a rough path over a multidimensional $B^H$ \cite{Unterberger:2010td},  for $H\leq1/4$ it is believed that there exists no cannonical lift into a geometric rough path. For $H>1/4$ this is possible, and it was even shown that under the cannonical lifting of $B^H$, the signature uniquely determines the sample paths \cite{Boedihardjo:2014td}. 

While it would be of interest to determine the value of the limit for this cannonical lift, for now we only focus on the case $H>1/2$ in the one-dimensional setting. Then we're able to define an Itô-type integral with a version of the Itô formula, and apply the results of the previous chapter. This integral is called  the fractional Wiener-Itô-Skorokhod (fWIS) integral and is defined analogously to the Skorokhod integral (following the reference \cite{Oksendal:2008td}). Recall the basic white noise theory set-up: 

Let $\Omega = \mathcal{S}'(\mathbb{R}^+)$ be the space of tempered distributions, defined as the dual of the Schwarz space $\mathcal{S}(\mathbb{R}^+)$ of rapidly decreasing functions. In the Brownian motion case we would use Bochner-Minlos theorem to construct a probability measure $\P$ on $\mathcal{B}(\Omega)$ characterised by:
\begin{align*}
    \int_{\Omega} \exp{(i\langle \omega, f \rangle)} \P(\d \omega)= \exp{\left(-\frac{1}{2}\| f \|_{L^2(\mathbb{R}^+)}\right)} \hspace{1cm} \forall f \in \mathcal{S}(\mathbb{R}^+)
\end{align*}
Then we would have that the random variable $F(\omega) := \langle \omega, f \rangle$ corresponds with the Itô integral of deterministic $f$ with respect to the Brownian motion. In the case of fractional Brownian motion, all we need to do is to change the inner product on $\mathcal{S}(\mathbb{R}^+) \subset L^2(\mathbb{R}^+)$ into:
\begin{align*}
    \langle f,g \rangle_H := \int_0^{\infty} \int_0^{\infty} f(s) g(t) \theta_H (s,t) \d s \d t, 
\end{align*}
where $\theta_H (s,t) = H(2h-1)|s-t|^{2h-1}$ so that  $\int_0^t \int_0^s \theta_H (u,v) \d u \d v = R_H(s,t)$. Let $L_{\theta,H}^2(\mathbb{R}^+)$ be the completion of the Schwarz space with respect to the norm induced by this new inner product, which becomes a Hilbert space. Then for any $f$ in this space, we can define the stochastic integral by $(\int_0^t f_s \d B^H_s) (\omega):= \langle \omega, \lim_n f_n1_{[0,t]} \rangle = \lim_n \langle \omega, f_n1_{[0,t]} \rangle$ or more constructively as in \cite{Grippenberg:1996td}. Note that for $\Tilde{B}_t^H(\omega) := \langle \omega, 1_{[0,t]} \rangle$ we have that 
\begin{align*}
    \E[\Tilde{B}_t^H] &= 0 \\
    \E[\Tilde{B}_t^H\Tilde{B}_s^H] &= \langle 1_{[0,s]},1_{[0,t]} \rangle_H = \int_0^t \int_0^s \theta_H (u,v) \d u \d v = R_H(s,t)
\end{align*}
From Kolgomorov's continuity criterion it follows that this process has a continuous version, which is a fractional Brownian motion, as we would expect. 

To define the fWIS integral for stochastic integrand and set up an analouge of the Itô formula a lot of work is required, using tools from Malliavin calculus, and so we refer the reader to \cite{Oksendal:2008td}. We only need the following corollary of Itô formula, which can be found in the reference as Theorem 3.8.1:

\begin{lemma}
    For deterministic $f\in L_{\theta,H}^2(\mathbb{R}^+)$ define:
\begin{align*}
    h_n^{H,f}(t) = \|f\|_{\theta,H, t}^n h_n\left( \|f\|_{\theta,H, t}^{-1} \int_0^t f_s dB_s^H \right),
\end{align*}
where:
\begin{align*}
    \|f\|_{\theta, H, t}^2 := \int_0^{t} \int_0^{t} f(u) f(v) \theta_H (u,v) \d u \d v, 
\end{align*}
and $h_n(x) = 2^{-n/2} \Tilde{h}_n(x/\sqrt{2})$ are another version of the Hermite polynomials as defined in Definition \ref{hermite}. This satisfies the following recursive relation: 
\begin{align*}
    dh_n^{H,f}(t) = n h_{n-1}^{H,f}(t)f_t dB_t^H
\end{align*}
\end{lemma}

Now we can follow the same strategy as taken in Section \ref{section-main} to show that iterated integrals with respect to the stochastic process $(X_t^H)_{t\geq0}$, driven by fractional BM in the sense that $\d X_t^H = f_t \d B_t^H$ for some $f\in L_{\theta,H}^2(\mathbb{R}^+)$, satisfy:
\begin{align*}
    \mathbb{X}^{n, H}_{0,t} := \int_0^t \int_0^{s_n} \dots \int_0^{s_2} \d X_{s_1}^H \dots \d X_{s_{n-1}}^H \d X_{s_{n}}^H = \frac{1}{n!} h_n^{H,f}(X_t^H)
\end{align*}
Consequently, 
\begin{align*}
L^{\mathbb{X}, H}_{0,t} :&= \limsup_n \left| (n/2)!\mathbb{X}^{n, H}_{0,t} \right|^{2/n} \\
&= \limsup_n \left| \frac{(n/2)!}{n!} h_n^{H,f}(X_t^H) \right|^{2/n} \\
&= \limsup_n \left| \frac{(n/2)!}{n!} \|f\|_{\theta,H, t}^n h_n\left( \|f\|_{\theta,H, t}^{-1} \int_0^t f_s dB_s^H \right) \right|^{2/n} \\
&= \frac{\|f\|_{\theta,H, t}^2}{2} \limsup_n \left| \frac{(n/2)!}{n!} \Tilde{h}_n\left( \frac{\|f\|_{\theta,H, t}^{-1}X_s^H }{\sqrt{2}}  \right) \right|^{2/n}
\end{align*}
We already know from the proof of Theorem \ref{main_thrm} that this limit converges to 1 almost surely and therefore we proved:
\begin{theorem}
     Let $(B_t^H)_{t\geq 0}$ be a fractional Brownian motion and $X_t^H := \int_0^t f_s \d B_s^H$ for some $f\in L_{\theta,H}^2(\mathbb{R}^+)$. Then almost surely: $L^{\mathbb{X}, H}_{0,t} = \frac{1}{2}\|f\|_{\theta,H, t}^2$
\end{theorem}
As a simple application of this result for $f = 1_{[0,t]}$ we get that 
\begin{align*}
    L^{\mathbb{B}, H}_{0,t} = \frac{1}{2}\|1_{[0,t]}\|_{\theta,H, t}^2 = \frac{1}{2}R_H(t,t) = \frac{1}{2}t^{2H}
\end{align*}
which agrees with the result we found for standard Brownian Motion ($H=1/2$). This can be used to derive an expression for $H$ in terms of iterated fWIS integrals (log is an increasing continuous function, so we can exchange it with lim sup):
\begin{corollary}
    Let $(B_t^H)_{t \geq 0}$ be a fractional Brownian motion. Then for any $t> 0$ we have almost surely:
    \begin{align*}
        H = \limsup_n \frac{1}{2\log t} \log \left( 2 \left| (n/2)!\mathbb{B}^{n, H}_{0,t} \right|^{2/n} \right)
    \end{align*}
\end{corollary}

\section{Extension into multi-dimensional settings and further work}\label{section-multidim}

To extend our results on the limit of Itô signatures for $d$-dimensional semimartingales, we first looked at $d$-dimensional Brownian motion and managed to prove an upper bound. 

\begin{theorem}
    Let $(B_t)_{t\geq 0} = (B_t^1, \dots , B_t^d)_{t\geq 0}$ be a $d$-dimensional Brownian motion and $\mathbb{B}_{0,t} \in T((\Rd))$ its Itô signature. Assume that $(\| \cdot \|_n)_n$ a sequence of $l^2$ norms, each on $(\Rd)^{\otimes n}$. Then almost surely we have: 
    \begin{align*}
        \limsup_n \left \|  \left(\frac{n}{2}\right)! \mathbb{B}_{0,t}^n \right \|_n^{\frac{2}{n}} \leq \frac{dt}{2} 
    \end{align*}
\end{theorem}

\begin{proof}
    Fix $k\in \mathbb{N}$. If we show that we can bound the limit from above by $\frac{dt}{2} + \frac{1}{k}$ outside some null set $\mathcal{N}_k$, then we can let $\mathcal{N} = \bigcup_k \mathcal{N}_k$, which is still a null set, and then the claim follows. It is enough to show that:
    \begin{align*}
        \P \left( \left \|  \left(\frac{n}{2}\right)! \mathbb{B}_{0,t}^n \right \|_n^{\frac{2}{n}} > \frac{dt}{2} + \frac{1}{k}  \text{ i.o.} \right) = 0
    \end{align*}
    By Borel-Cantelli, it suffices checking that the following sum converges: 
    \begin{align*}
        \sum_{n=1}^{\infty} \P \left( \left(\left(\frac{n}{2}\right) !\right)^2 \left \| \mathbb{B}_{0,t}^n \right \|_n^2 > \left(\frac{dt}{2} + \frac{1}{k}\right)^{n} \right)
    \end{align*}
    Markov inequality gives us the following upper bound:
    \begin{align*}
        \sum_{n=1}^{\infty} \left(\frac{dt}{2} + \frac{1}{k}\right)^{-n} \E \left[ \left(\left(\frac{n}{2}\right) !\right)^2 \left \| \mathbb{B}_{0,t}^n \right \|_n^2 \right]
    \end{align*}
    Now if we manage to show that 
    \begin{align}\label{eq:6}
        \lim_n \E \left[ \left(\left(\frac{n}{2}\right) !\right)^2 \left \| \mathbb{B}_{0,t}^n \right \|_n^2 \right]^{\frac{1}{n}} \leq \frac{dt}{2} 
    \end{align}
    we can find $N$ such that for $n>N$ the LHS is less that $\frac{dt}{2} + \frac{1}{2k}$ and then the sum above converges: 
    \begin{align*}
        \sum_{n=1}^{\infty} \left(\frac{dt}{2} + \frac{1}{k}\right)^{-n} \E \left[ \left(\left(\frac{n}{2}\right) !\right)^2 \left \| \mathbb{B}_{0,t}^n \right \|_n^2 \right] \leq \sum_{n=1}^N (\dots) + \sum_{n=N+1}^{\infty} \left(\frac{\frac{dt}{2} + \frac{1}{2k}}{\frac{dt}{2} + \frac{1}{k}}\right)^{-n} < \infty
    \end{align*}
    We actually manage to show equality for (\ref{eq:6}). Note that asymptotically, $\left(\left(\frac{n}{2}\right) !\right)^2 \sim \frac{n!}{2^n}$. Then the limit is equal to: $\frac{1}{2} \lim_n \left(n! \E \| \mathbb{B}_{0,t}^n  \|_n^2 \right)^{1/n}$, and so it suffices to show that $\E \| \mathbb{B}_{0,t}^n  \|_n^2 = \frac{d^nt^n}{n!}$. We'll do this by induction on $n$. For $n=1$ we have by Itô isometry that: 
    \begin{align*}
        \E \| \mathbb{B}_{0,t}^1  \|_1^2 = \sum_{i=1}^d \E \left[\left( \int_0^t \d B_s^i \right)^2
        \right] = dt   
    \end{align*}
    Assuming the claim for $n-1$, we now have:
    \begin{align*}
        \E \| \mathbb{B}_{0,t}^n  \|_n^2 &= \sum_{i_1=1, \dots, i_n = 1}^d \E \left[\left( \int_0^t \dots \int_0^{s_2} \d B_{s_1}^{i_1} \dots \d B_{s_n}^{i_n} \right)^2 \right] \\
        &= \sum_{i_n = 1}^d \int_0^t   \sum_{i_1=1, \dots, i_{n-1} = 1}^d \E \left[\left( \int_0^{s_n} \dots \int_0^{s_2} \d B_{s_1}^{i_1} \dots \d B_{s_{n-1}}^{i_{n-1}} \right)^2\right] \d s_n \\
        &= d \int_0^t   \frac{d^{n-1} s_n^{n-1}}{(n-1)!} \d s_n = \frac{d^nt^n}{n!},
    \end{align*}
    which concludes the proof.
\end{proof}

\subsection{Further work and conclusion}

The first step in extending this into general semimartingales would be looking at a process, whose each coordinate is a Brownian motion scaled by a different constant. In this case we could show that the limit is bounded by the square of the highest constant, but that's not sharp.

Proving a lower bound is even harder. In \cite{Boedihardjo:2019td} a lower bound for Stratonovich signatures was proved using hyperbolic development (introduced in \cite{Hambly:2010td}), but this crucially relied on terms that stem from Itô-Stratonovich conversion, and so we couldn't apply it in our setting. Utilising formulas for the expected signature (e.g. \cite{Fawcett:2002td}) failed for similar reasons.

Extending the methods of Remark \ref{rem} wouldn't work as the "clock" of each coordinate runs at a different speed, so we can't change the time uniformly. Deriving a generalization of Hermite polynomials in higher dimension might be useful, but now it's unclear what SDE we'd like it to solve - how do we even wish define the quadratic variation for a $d$-dimensional process?

Another method we unsuccessfully tried was to use that the signature kernel solves a Goursat PDE \cite{Salvi:2021td} and then exploit properties of this PDE. 

Finally, we ran simulations of a 2-dimensional Brownian motion, whose coordinates are scaled by different factors, and computed signature of its piecewise linear approximations, which should converge to the Stratonovich signature. These simulations seemed to suggest the following:
\begin{conjecture}
    Let $(B_t)_{t\geq 0} = (B_t^1, \dots, B_t^d)_{t\geq 0}$ be a $d$-dimensional Brownian motion and $V$ an invertible $d\times d$ matrix. Then for $X = VB$ we have $L^{\mathbb{X}, \text{Strat}}_{s,t} = |\det V|\kappa_d (t-s)$, where $\kappa_d$ only depends on the dimension $d$ and choice of tensor norms.
\end{conjecture}
The precise value and meaning of the constant $\kappa_d$ is still unclear; according to our simulations in $d=2$ it should be around 1.3 for $l^1$ and 0.7 for $l^2$ norms. 

What is worse, convergence to the constant was really slow for some sample paths, and computational cost of calculating higher order terms of the signature explodes exponentially fast. Even if it didn't, results of this paper still can't be used for estimation of quadratic variation or the Hurst parameter, because we yet can't recover Itô signature from the signature of piecewise linear approximations. 

We managed to prove a simple formula for this conversion in $d=1$: 
\begin{align*}
    \mathbb{B}_{0,t}^{\text{Itô}} = \mathbb{B}^{\text{Strat}}_{0,t} \otimes \exp \left(-\frac{1}{2}e_1 \otimes e_1 t\right)
\end{align*}
However, there is still very little we know about the higher-dimensional setting, which would require further work.

\bibliographystyle{alpha}
\bibliography{references} 

\end{document}